\documentclass[reqno]{amsart}
\usepackage{amsmath,amssymb}
\usepackage[mathscr]{euscript}
\usepackage{mathtools}
\usepackage{xcolor}
\usepackage[msc-links]{amsrefs}
\usepackage{dsfont}
\usepackage{relsize}
\usepackage[normalem]{ulem}
\usepackage{bbold}
\usepackage[draft]{fixme}
\usepackage{hyperref}     

\definecolor{bblue}{rgb}{.2,0.2,.8}

\theoremstyle{plain}
\newtheorem{theorem}{Theorem}[section]
\newtheorem{proposition}[theorem]{Proposition}
\newtheorem{lemma}[theorem]{Lemma}
\newtheorem{corollary}[theorem]{Corollary}

\theoremstyle{definition}
\newtheorem{definition}[theorem]{Definition}

\theoremstyle{remark}
\newtheorem{remark}[theorem]{Remark}
\newtheorem{example}[theorem]{Example}

\numberwithin{equation}{section}
\numberwithin{theorem}{section}

\newcommand*{\defeq}{\mathrel{\vcenter{\baselineskip0.5ex \lineskiplimit0pt
			\hbox{\scriptsize.}\hbox{\scriptsize.}}}%
	=}

\newcommand{\circbin}{\mathbin{\circ}}
\newcommand{\oneplus}[1]{\omega_{#1}}

\renewcommand{\epsilon}{\varepsilon}

\newcommand{\N}{\mathbb{N}}

\newcommand{\mB}{\mathcal{B}}

\newcommand{\Set}[1]{\left\{#1\right\}}

\DeclareMathOperator{\Id}{I}

\DeclareMathOperator{\ord}{ord}
\DeclareMathOperator{\End}{End}
\title{Near-Rings and Skew Braces}

\author[R. Aragona]{Riccardo Aragona}
\address{\noindent Riccardo Aragona \hfill\break\indent 
	Department of Information Engineering, Computer Science and Mathematics\hfill\break\indent  University of L'Aquila
	\hfill\break\indent 
	67100 Coppito, L'Aquila, Italy
}
\email{riccardo.aragona@univaq.it}

\author[N. Gavioli]{Norberto Gavioli}
\address{\noindent Norberto Gavioli \hfill\break\indent 
	Department of Information Engineering, Computer Science and Mathematics\hfill\break\indent  University of L'Aquila
	\hfill\break\indent 
	67100 Coppito, L'Aquila, Italy
}
\email{norberto.gavioli@univaq.it}

\author[M. Iannaccone]{Martina Iannaccone}
\address{\noindent Martina Iannaccone \hfill\break\indent 
	Department of Information Engineering, Computer Science and Mathematics\hfill\break\indent  University of L'Aquila
	\hfill\break\indent 
	67100 Coppito, L'Aquila, Italy
}
\email{martina.iannaccone@graduate.univaq.it}

\author[G. Nozzi]{Giuseppe Nozzi}
\address{\noindent Giuseppe Nozzi \hfill\break\indent 
	Department of Information Engineering, Computer Science and Mathematics\hfill\break\indent  University of L'Aquila
	\hfill\break\indent 
	67100 Coppito, L'Aquila, Italy
}
\email{giuseppe.nozzi@graduate.univaq.it}

\begin{document}

\thanks{All the authors are members of INdAM--GNSAGA}
\begin{abstract}
	This paper provides a deeper analysis of Rump's correspondence between near-rings and skew braces. We show that, under a topological assumption on the structure of a near-ring,  Rump's correspondence is invertible and it is always possible to construct an associated skew brace. Furthermore, we provide  several examples where this topological assumption is fulfilled in well-known algebraic structures such as the Nottingham group, groups of triangular functions, iterated wreath products of arbitrary groups, and groups of IA-automorphisms of free nilpotent groups.
\end{abstract}

\noindent

\maketitle
\thispagestyle{empty}

\section{Introduction}\label{sec:intro}

In his seminal paper~\cite{Rump2007}, Rump introduced the notion of a \emph{brace},
exhibiting a deep connection between these algebraic structures and
set-theoretic solutions of the Yang--Baxter equation, whose formulation in a set-theoretic framework goes back to Drinfeld~\cite{Drinfeld}
and was first investigated in~\cites{Gateva, Etingof, Cedo}. A
cornerstone of that theory is a bijective correspondence between
two-sided braces and radical rings: given a ring $(R,+,\cdot)$ that is
radical with respect to the circle operation $a\circbin b = a+b+ab$, the
pair $(R,+,\circbin)$ carries a natural brace structure. Since then, braces
--- and their non-abelian generalization, \emph{skew braces},
see e.g.\ Guarnieri and Vendramin~\cite{Guarnieri2017} and Vendramin's recent survey~\cite{Vendramin2024} ---
have become a central tool for constructing and classifying solutions
of the Yang--Baxter equation.

Given a right near-ring
$(R,+,\circbin)$ with multiplicative identity $\Id$ (see \cite{meldrum}), one may define
on $R$ the operation
\[
f \ast g \defeq g+f\circbin (\Id +g)
\]
which is always associative and turns $R$ into a monoid with identity
$0$. In Section~\ref{sec:general} we isolate a
set of natural compatibility conditions --- expressed in terms of a
filtration $\{T_i\}_{i\ge1}$ of a subgroup $T$ of $R$  --- under which
$(T,\ast)$ is in fact a topological group, and $(T,+,\ast)$ is a topological right
skew brace (see~\cite{Mondal2026}). The
key technical device is a difference operator
$\Delta_g(f) \defeq f\circbin(\Id+g)-f$, measuring the failure of
the left distributive law; its interaction with
the filtration is what forces convergence of the $\ast$-inverses and
continuity of both operations.

The relation between near-rings and skew braces was first investigated by Rump in~\cite{rumpnear}.\ Every right near-ring
with identity carries a \emph{skew-ring} structure, and every skew brace arises from a suitable near-ring via an
exponential homomorphism~\cite[Theorem~1]{rumpnear}.\ However, the exponential construction yields in general only a
monoid with respect to \(\ast\), which need not be a group. Rump shows that the group property holds when the near-ring
is local~\cite[Example~8]{rumpnear}, covering examples such as the Nottingham group over a field and the finite
triangular functions over an elementary abelian \(p\)-group. In this paper, we show that locality can be replaced by
suitable completeness and compatibility conditions with respect to a filtration. These conditions ensure that the
adjoint monoid \((\Id+T,\circbin)\) is a group and hence that \((T,+,\ast)\) is a topological right skew brace. 
This allows us to treat further examples, including the Nottingham group over \(\mathbb{Z}\), groups of IA-automorphisms 
of free nilpotent groups, and iterated wreath products of arbitrary finite groups.

Section~\ref{sec:nottingham} deals with the near-ring of formal
power series $D[[x]]$ over a commutative domain $D$ under composition. The $x$-adically
topological near-ring $T = x^2D[[x]]$ is shown to be compatible, yielding a skew brace whose multiplicative group is 
isomorphic to the classical Nottingham group~\cite{Camina2000}.

In Section~\ref{sec:triangular} we turn to a second, purely combinatorial family of
examples: \emph{triangular functions} on (possibly infinite) iterated products
of an abelian group $A$, i.e., vector functions whose $h$-th coordinate
depends only on the preceding ones. We refer to Klimov and Shamir~\cite{Klimov2002} for their cryptographic 
applications, under the name of \emph{T-functions}. We show that finite, finitary, positive, and negative triangular 
functions each carry a natural near-ring structure under composition, fitting into our
general framework via appropriate filtrations. In the finite, finitary and
positive cases, this yields new families of (topological) right braces,
closely related to iterated wreath products~\cites{Baumslag1959,Aragona2026,Aragona2026a, kal1, kal2,kal3,meldrum},
while the negative case produces a monoid that need not be a group with respect to $\ast$,
exhibiting a natural obstruction to invertibility.

Finally, in Sections~\ref{sec:wreath} and~\ref{sec:ia-endo} we explore the non-abelian setting. We
first show that iterated wreath products of arbitrary groups fit into our
construction, yielding topological skew braces with non-abelian additive groups and a
discrete topology. We then treat a more substantial example arising
from free nilpotent groups. Specifically, we identify the set of IA-endomorphisms (see~\cite{Bachmuth1965})
of a free nilpotent group $F_{n,c}$ with $\Id+T$ for a suitable \emph{left} near-ring $T$ of endomorphisms. Using a 
result of Andreadakis~\cite{Andreadakis1965}, we show that $T$ is compatible, so that $(T,+,\ast)$ is a topological left skew brace ---
realizing the group of IA-automorphisms of  $F_{n,c}$ as the multiplicative group of a skew brace with a discrete topology.

	\section{Preliminaries}
	In this section we collect the basic notions on skew braces (see
	e.g.~\cites{Guarnieri2017,Vendramin2024}) and near-rings (see e.g.~\cite{Meldrum1985}) that we shall use
	throughout the paper.
	
	\begin{definition}\label{def:near-ring}
		A \emph{right near-ring} is a triple $(R,+,\circbin)$, where $(R,+)$ is a
		(not necessarily abelian) group with identity element $0$,
		$(R,\circbin)$ is a semigroup, and the right distributive law
		\[
		(f+g)\circbin h=f\circbin h+g\circbin h
		\]
		holds for all $f,g,h\in R$. If $(R,\circbin)$ is a monoid, we say that
		$R$ is a right near-ring \emph{with identity}, and we denote the identity
		element of $(R,\circbin)$ by $\Id$.
	\end{definition}
	
	We shall refer to $\circbin$ as the \emph{composition}. Note that
	$0\circbin f=0$ for every $f\in R$, since
	$0\circbin f=(0+0)\circbin f=0\circbin f+0\circbin f$; on the other hand,
	$f\circbin 0$ need not be $0$. The prototypical example is the set $M(G)$ of
	all maps from a group $(G,+)$ to itself, with pointwise addition and
	composition of maps.
	
	We will denote by \(R_c\) the set \(\{r\in R\mid r\circbin s=r\ \text{for all }s\in R\}\) of constant elements of \(R\). Equivalently, \(R_c\) can be seen as the set \(\{ r\circbin 0 \mid r \in R\}\).
	
	\begin{definition}\label{def:skew-brace}
		A \emph{right skew brace} is a set $B$ endowed with two operations $+$
		and $\ast$ such that $(B,+)$ and $(B,\ast)$ are groups and the
		compatibility condition
		\[
		(a+b)\ast c=a\ast c-c+b\ast c
		\]
		holds for all $a,b,c\in B$. If $(B,+)$ is abelian, we say that $B$ is a
		right \emph{brace} (of abelian type).
	\end{definition}
	
	The group $(B,+)$ is called the \emph{additive group} and $(B,\ast)$ the
	\emph{multiplicative group} of the skew brace. The two groups share the same
	identity element: taking $a=b=0$ in the compatibility condition yields
	$0\ast c=0\ast c-c+0\ast c$, hence $0\ast c=c$ for every $c\in B$, so that
	$0$ is the identity of $(B,\ast)$.

	\begin{example}\label{ex:radical}
		If $(R,+,\cdot)$ is a radical ring, then $R$ with the operation
		$a\ast b\mathrel{:=}a+b+a\cdot b$ is a brace of abelian type~\cite{Rump2007}.
	\end{example}
	
	\begin{definition}[See~\cite{Mondal2026}]\label{def:top-skew-brace}
		A \emph{topological right skew brace} is a right skew brace $B$ endowed
		with a topology such that both $(B,+)$ and $(B,\ast)$ are topological
		groups.
           \end{definition}

\section{A general construction}\label{sec:general}
In this section we provide a general construction of skew braces starting
from near-rings. This procedure generalizes the
	construction of braces from radical rings introduced in
	\cite{Rump2007}.

Throughout, $(R,+,\circbin)$ is a \emph{right} near-ring with identity
$\Id$, i.e.\ $(R,+)$ is a (not necessarily abelian) group, $(R,\circbin)$ is
a monoid with identity $\Id$, and the right distributive law
\[(a+b)\circbin c=a\circbin c+b\circbin c\]
holds for all $a,b,c\in R$. We write $a-b\defeq a+(-b)$. Note that
$0\circbin a=0$ for every $a\in R$, since
$0\circbin a=(0+0)\circbin a=0\circbin a+0\circbin a$.

Let $T$ be a subgroup of $(R,+)$ and let $\{T_i\}_{i\ge1}$ be a filtration of \(T\), i.e.\ a sequence of subgroups of
\((T,+)\) satisfying
\begin{equation}\label{eq:filtration}
	T=T_1\supseteq T_2\supseteq\cdots
	\quad\text{and}\quad \bigcap_{i\ge1}T_i=\{0\}.
\end{equation}
Furthermore, we require that for all \(i\ge 1\)
\begin{gather}
	T_i \trianglelefteq (T,+),	\label{eq:normality}\\
	T_i\circbin(\Id+T)\subseteq T_i	\label{eq:ast-closure-Ti}.
\end{gather}
Note that condition~\eqref{eq:normality} is automatic when $(T,+)$ is abelian.

The filtration $\{T_i\}_{i\ge1}$ induces a topology on $T$ by taking the
subgroups $T_i$ as a fundamental system of neighborhoods of $0$. For \(t,t'\in T\) with \(t\neq t'\), set \(s(t,t')=\max\{i
\mid -t+t'\in T_i\}\) (the maximum exists since \(\bigcap_{i\ge 1} T_i=\{0\}\)). This
topology is induced by the distance
\[
d(t,t')=\begin{cases}
	\exp(-s(t,t')) & \text{if } t\ne t',\\
	0        & \text{otherwise}.
\end{cases}
\]
\begin{lemma}\label{lem:metric}
	The distance $d$ on $(T,+)$ satisfies
	\begin{enumerate}
		\item $d(t,t')=d(t',t)$ and
		$d(t,t')\le\max\{d(t,z),d(z,t')\}$;
		\item $d(z+t,z+t')=d(t,t')=d(t+z,t'+z)$
	\end{enumerate}
	for all $t,t',z\in T$.
\end{lemma}

\begin{proof} 
	Symmetry and the ultrametric inequality both follow
		from the fact that each $T_i$ is a subgroup; for the latter, write
		$-t+t'=(-t+z)+(-z+t')$.
	
	The left invariance is immediate since $-(z+t)+(z+t')=-t+t'$. The right invariance follows by the normality of
	\((T_i,+)\) in \((T,+)\); indeed, the element
	\(
	-(t+z)+(t'+z)=-z+(-t+t')+z
	\)
	lies in \(T_i\) if and only if \(-t+t'\in T_i\).
\end{proof}
It immediately follows that the operation \(+\) is continuous in the topology induced by the filtration $\{T_i\}_{i\ge1}$; moreover, since $d$ is bi-invariant by Lemma~\ref{lem:metric}, inversion is an isometry.
In particular $(T,+)$ is a topological group in which every $T_i$ is a
clopen subgroup.

From now on we assume that $T$ is complete with respect to the metric $d$.
This is the case when the filtration is finite, i.e.\ $T_n=\{0\}$ for some
$n$, since the induced topology is the discrete one.

The following definition will be useful to construct a skew brace structure on $T$. We point out that the
set $T$ is closed under the maps $f\mapsto f\circbin(\Id+g)$ for $g\in T$. This is a consequence
of~\eqref{eq:ast-closure-Ti} with $i=1$.

\begin{definition}
	Let $f,g\in R$. We define the \emph{difference operator}
	$\Delta_g\colon R\to R$ by
	\[
	\Delta_g(f)\defeq f\circbin(\Id+g)-f .
	\]
	More generally, for $f,f',g\in R$ we set
	\[
	\Delta_g(f,f')\defeq f\circbin(\Id+f'+g)-f\circbin(\Id+f'),
	\]
	so that $\Delta_g(f)=\Delta_g(f,0)$.
\end{definition}

The following equality holds for all \(f,f',g\in R\) and we shall use this fact without further mention:
\[
\Delta_g(f,f')=\Delta_{f'+g}(f)-\Delta_{f'}(f).
\]
\begin{remark}\label{rem:iterated-delta}
	If \((T,+)\) is abelian, by iterating the difference operator we obtain that, for all \(f,g\in T\) and
	\(n\in \N\), the following equality holds
	\begin{equation}\label{eq:delta}
		\Delta^{n}_g(f)= \sum_{i=0}^{n} (-1)^i \binom{n}{i} f\circbin (\Id+g)^{n-i},
	\end{equation}
	where \(\Delta^n_g(f)=\Delta_g(\Delta_g^{n-1}(f))\), \(\Delta^0_g(f)=f\), and the powers of $(\Id+g)$ are taken with respect to $\circbin$.
\end{remark}
We introduce a new operation \(\ast\) on \(R\). For \(f,g\in R\) we define
\begin{equation}\label{eq:ast}
	f\ast g\defeq g+f\circbin(\Id+g),
\end{equation}
which is the unique element satisfying
\begin{equation}\label{eq:ast-via-Id}
	\Id+(f\ast g)=(\Id+f)\circbin(\Id+g).
\end{equation}
Equivalently, using the difference operator we can write $f\ast g=g+\Delta_g(f)+f$.
\begin{proposition}\label{prop:monoid}
	$(T,\ast)$ is a monoid with identity element $0$.
\end{proposition}
\begin{proof}
	By Condition~\eqref{eq:ast-closure-Ti} with $i=1$ we have
	$f\circbin(\Id+g)\in T$ for all $f,g\in T$, hence $f\ast g\in T$. The map
	$f\mapsto\Id+f$ is a bijection $T\to\Id+T$ and
	by~\eqref{eq:ast-via-Id} it intertwines $\ast$ with $\circbin$, so that the
		associativity of $\ast$ follows from the associativity of $\circbin$.
	Finally $f\ast0=0+f\circbin\Id=f$ and
	$0\ast f=f+0\circbin(\Id+f)=f+0=f$.
\end{proof}
We now introduce a compatibility condition between the difference operator and the filtration, which ensures that
\((T,\ast)\) is a topological group.

\begin{definition}\label{def:compatible}
	We say that the difference operator $\Delta$ is \emph{compatible} with
	the filtration $\{T_i\}_{i\ge1}$ if
	\[
	\Delta_g(f,f')\in T_{i+1}
	\]
	for all $f,f'\in T$ and $g\in T_i$. In this case we say that $T$ is
	\emph{compatible}.
\end{definition}

\begin{proposition}\label{prop:phi-continuous}
	If $T$ is compatible, then the map $\varphi\colon T\times T\to T$,
	$\varphi(f,g)=f\circbin(\Id+g)$, is continuous.
\end{proposition}

\begin{proof}
	Let $i\ge1$ and let $(f,g),(f',g')\in T\times T$ with
	$\max\{d(f,f'),d(g,g')\}\le\exp(-i)$. Then $f'=f+\delta_f$ and
	$g'=g+\delta_g$ with $\delta_f,\delta_g\in T_i$. We have that
	\begin{align*}
		\varphi(f',g')
		&=(f+\delta_f)\circbin(\Id+g+\delta_g)\\
		&=f\circbin(\Id+g+\delta_g)+\delta_f\circbin(\Id+g')\\
		&=\Delta_{\delta_g}(f,g)+\varphi(f,g)+\delta_f\circbin(\Id+g').
	\end{align*}
	Hence
	\[
	-\varphi(f,g)+\varphi(f',g')
	=\bigl(-\varphi(f,g)+\Delta_{\delta_g}(f,g)+\varphi(f,g)\bigr)
	+\delta_f\circbin(\Id+g').
	\]
	By the compatibility of $T$, we have $\Delta_{\delta_g}(f,g)\in T_{i+1}\subseteq T_i$, which implies that its
	conjugate \(-\varphi(f,g)+\Delta_{\delta_g}(f,g)+\varphi(f,g)\) also lies in $T_i$ by~\eqref{eq:normality}.
	Since $\delta_f\circbin(\Id+g')\in T_i$ by~\eqref{eq:ast-closure-Ti}, it follows that
	$d(\varphi(f,g),\varphi(f',g'))\le\exp(-i)$.
\end{proof}

\begin{corollary}\label{cor:ast-continuous}
	If $T$ is compatible, then
	$\ast\colon T\times T\to T$ is continuous.
\end{corollary}

\begin{proof}
	It is an immediate consequence of Lemma~\ref{lem:metric} and Proposition~\ref{prop:phi-continuous}.
\end{proof}

\begin{proposition}\label{prop:Tastgroup}
	If $T$ is compatible, then $(T,\ast)$ is a group.
\end{proposition}
\begin{proof}
	By Proposition~\ref{prop:monoid}, it suffices to construct a right $\ast$-inverse
	for every $f\in T$, since a monoid in which every element admits a
		right inverse is a group.
	
	Fix $f\in T$ and set $g_1\defeq-f$. Setting \(\delta_2\defeq f\ast g_1\), we have
	\[
	\delta_2=-f+f\circbin(\Id-f)
	=-f+\Delta_{-f}(f)+f\in T_2,
	\]
	since $\Delta_{-f}(f)\in T_2$ by compatibility and $T_2$ is normal in $(T,+)$.
	
	Proceeding inductively, suppose we have found $g_j\in T$ such that
	$\delta_{j+1}\defeq f\ast g_j\in T_{j+1}$. We then define
	\[
	g_{j+1}\defeq g_j\ast(-\delta_{j+1}).
	\]
	By Equation~\eqref{eq:ast-via-Id}, it follows that
	\[
	\Id+(f\ast g_{j+1})
	=(\Id+f)\circbin(\Id+g_j)\circbin(\Id-\delta_{j+1})
	=(\Id+\delta_{j+1})\circbin(\Id-\delta_{j+1}),
	\]
	which yields
	\[
	\delta_{j+2}=f\ast g_{j+1}
	=\delta_{j+1}\ast(-\delta_{j+1})
	=-\delta_{j+1}+\Delta_{-\delta_{j+1}}(\delta_{j+1})+\delta_{j+1}
	\in T_{j+2},
	\]
	where we have applied the compatibility of \(T\) along with~\eqref{eq:normality}.
	Next, observe that
	\[
	g_{j+1}=g_j\ast(-\delta_{j+1})
	=-\delta_{j+1}+\Delta_{-\delta_{j+1}}(g_j)+g_j,
	\]
	which implies
	$g_{j+1}-g_j=-\delta_{j+1}+\Delta_{-\delta_{j+1}}(g_j)\in T_{j+1}$.
	Thus, by~\eqref{eq:normality}, we also have $-g_j+g_{j+1}\in T_{j+1}$.
	For $i<j$, the telescoping identity
	\[
	-g_i+g_j=(-g_i+g_{i+1})+\cdots+(-g_{j-1}+g_j)\in T_{i+1}
	\]
	shows that $d(g_i,g_j)\le\exp(-(i+1))$, meaning that the sequence $\{g_j\}_{j\ge1}$
	is Cauchy. By completeness, the limit $g\defeq\lim_{j\to\infty}g_j$ exists.
	Finally, Corollary~\ref{cor:ast-continuous} ensures that
	\[
	f\ast g=f\ast\lim_{j\to\infty}g_j
	=\lim_{j\to\infty}(f\ast g_j)
	=\lim_{j\to\infty}\delta_{j+1}=0. \qedhere
	\]
\end{proof}
From now on, we denote by $(f)^{-1}_\ast$ the inverse of $f\in T$ in the group $(T,\ast)$.
\begin{remark}\label{rem:inverseformula}
	Assume that $T$ is compatible and $(T,+)$ is abelian. Whenever
	$\lim_{i\to\infty}\Delta^i_f(f)=0$, the $\ast$-inverse of $f$ is given
	by the closed formula
	\[
	s(f)=\sum_{i=0}^{\infty}(-1)^{i+1}\Delta^i_f(f),
	\]
	since, by the continuity and additivity of \(\Delta_f\) (the latter
		following from the right distributive law) and completeness of \(T\), we have that
	\begin{align*}
		s(f)\ast f&= f+s(f)+\Delta_f(s(f))\\
		&= f+s(f)+\sum_{i\ge0}(-1)^{i+1}\Delta^{i+1}_f(f)\\
		&= f+s(f)-f-s(f)=0.
	\end{align*}
\end{remark}

\begin{proposition}\label{prop:cosets}
	Assume that $T$ is compatible. Then for every $i\ge1$:
	\begin{enumerate}
		\item $T_i$ is a normal subgroup of $(T,\ast)$;
		\item $f\ast T_i=f+T_i$ for every $f\in T$.
	\end{enumerate}
\end{proposition}

\begin{proof}
	(1) Consider $f,g\in T_i$. We have that $f\ast g=g+f\circbin(\Id+g)\in T_i$
	by~\eqref{eq:ast-closure-Ti}, so $T_i$ is closed under $\ast$.
	Now consider the sequences $\{g_j\}$,
	$\{\delta_{j+1}\}$ of Proposition~\ref{prop:Tastgroup}. We claim that
	$g_j\in T_i$ for every $j$. Indeed $g_1=-f\in T_i$; and if
	$g_j\in T_i$, then $\delta_{j+1}=f\ast g_j\in T_i$ since $T_i$ is
	$\ast$-closed, hence $-\delta_{j+1}\in T_i$ and
	$g_{j+1}=g_j\ast(-\delta_{j+1})\in T_i$. Finally, since \(T_i\) is clopen and
	$g_j\to(f)^{-1}_\ast$, we have that
	$(f)^{-1}_\ast\in T_i$, showing that \((T_i, \ast)\) is a subgroup of \((T, \ast)\).
	
	Now suppose that $f\in T$, $g\in T_i$. Let $h=(f)^{-1}_\ast$.
	By Equation~\eqref{eq:ast-via-Id},
	\begin{align*}
		\Id+(h\ast g\ast f)
		&=(\Id+h)\circbin(\Id+g)\circbin(\Id+f)\\
		&=\Id+f+g\circbin(\Id+f)+h\circbin\bigl(\Id+f+g\circbin(\Id+f)\bigr).
	\end{align*}
	We set $w\defeq g\circbin(\Id+f)\in T_i$
	(by~\eqref{eq:ast-closure-Ti}). Then
	\[
	h\circbin(\Id+f+w)=\Delta_w(h,f)+h\circbin(\Id+f)
	=\Delta_w(h,f)-f,
	\]
	using $h\circbin(\Id+f)=-f$. Therefore
	\[
	h\ast g\ast f=f+\bigl(w+\Delta_w(h,f)\bigr)-f\in T_i,
	\]
	by compatibility and Condition~\eqref{eq:normality}.
	
	(2) Let $\delta\in T_i$. Then we obtain that
	$f\ast\delta=\delta+\Delta_\delta(f)+f$ with
	$\Delta_\delta(f)\in T_{i+1}\subseteq T_i$, and so
	$f\ast\delta\in T_i+f=f+T_i$ by~\eqref{eq:normality}. It follows that
	$f\ast T_i\subseteq f+T_i$. Conversely, for $\delta\in T_i$ setting
	$\delta'\defeq(f)^{-1}_\ast\ast(f+\delta)$ we get that
	\begin{align*}
		\delta'&=(f+\delta)+(f)^{-1}_\ast\circbin(\Id+f+\delta)\\
		&=f+\delta+\Delta_\delta((f)^{-1}_\ast,f)+(f)^{-1}_\ast\circbin(\Id+f)\\
		&=f+\bigl(\delta+\Delta_\delta((f)^{-1}_\ast,f)\bigr)-f\in T_i ,
	\end{align*}
	by compatibility and~\eqref{eq:normality}. Since
	$f\ast\delta'=f+\delta$, we get $f+T_i\subseteq f\ast T_i$.
\end{proof}
In particular, the additive and the multiplicative coset spaces of $T_i$ coincide and we have the following result.

\begin{corollary}\label{cor:T-top-group}
	If \(T\) is compatible, then the filtration \(\{T_i\}_{i\ge 1}\) induces
		the same topology on $(T,+)$ and on $(T,\ast)$; in particular, $(T,\ast)$ is a
		topological group.
\end{corollary}

	\begin{proof}
		By Proposition~\ref{prop:cosets}, the subgroups $T_i$ are open and normal in
		$(T,\ast)$ and their $\ast$-cosets coincide with the additive cosets, hence
		they form a basis of the filtration topology for both groups. In particular,
		$(T,\ast)$ is a topological group.
\end{proof}

\begin{theorem}\label{thm:Tbrace}
	If $T$ is compatible, then $(T,+,\ast)$ is a topological right skew brace.
\end{theorem}
\begin{proof}
	By Lemma~\ref{lem:metric} and Corollary~\ref{cor:T-top-group}, both
	$(T,+)$ and $(T,\ast)$ are topological groups with respect to the same filtration
	topology. It remains to verify the skew brace compatibility condition. For all \(f,g,h\in T\) we have that
	\begin{align*}
		(f+g)\ast h
		&=h+(f+g)\circbin(\Id+h)\\
		&=h+f\circbin(\Id+h)+g\circbin(\Id+h)\\
		&=h+\bigl(-h+(f\ast h)\bigr)+\bigl(-h+(g\ast h)\bigr)\\
		&=(f\ast h)-h+(g\ast h). \qedhere
	\end{align*}
\end{proof}
\begin{remark}\label{rem:left_near_ring}
	The whole construction of this section can be carried out when $(R,+,\circbin)$ is a \emph{left} near-ring, i.e.\ when the
	left distributive law $a\circbin(b+c)=a\circbin b+a\circbin c$ holds. In
	this case $a\circbin 0=0$ for every $a\in R$, and all the definitions have
	to be mirrored: the condition $T_i\circbin(\Id+T)\subseteq T_i$ becomes
	\[
	(\Id+T)\circbin T_i\subseteq T_i,
	\]
	the difference operators become
	\[
	\Delta_g(f)\defeq(\Id+g)\circbin f-f,
	\qquad
	\Delta_g(f,f')\defeq(\Id+f'+g)\circbin f-(\Id+f')\circbin f,
	\]
	and the operation $\ast$ is defined by
	\[
	f\ast g\defeq f+(\Id+f)\circbin g,
	\]
	which is again the unique element satisfying
	$\Id+(f\ast g)=(\Id+f)\circbin(\Id+g)$. With these changes, all the
	arguments of this section carry over verbatim, and one obtains that, if
	$T$ is compatible, then $(T,+,\ast)$ is a topological \emph{left} skew brace.
\end{remark}

\bigskip

We now illustrate the connections between our construction and Rump's correspondence between near-rings and skew braces~\cite{rumpnear}. Following Rump, a \emph{skew-ring} is an additive group \((A,+)\) endowed with a monoid structure \((A,\circbin)\) such that \((a+b)\circbin c=a\circbin c-c+b\circbin c\) for all \(a,b,c\in A\), 
so that a skew-ring is a skew brace if and only if its adjoint monoid is a group. Every right near-ring \(R\) with
identity gives rise to the \emph{multiplicative skew-ring} \(S_\times R\), whose additive group is \((R,\oplus)\), with
\(a\oplus b\defeq a-\Id+b\), and whose adjoint monoid is \((R,\circbin)\). We have \((\Id+f)\oplus(\Id+g)=\Id+(f+g)\),
the \(\oplus\)-inverse of \(\Id+f\) is \(\Id-f\), and \((\Id+f)\circbin(\Id+g)=\Id+(f\ast g)\) by
Equation~\eqref{eq:ast-via-Id}. Hence \(\Id+T\) is a sub-skew-ring of \(S_\times R\), and the map \(f\mapsto\Id+f\) is
an isomorphism of \((T,+,\ast)\) onto \((\Id+T,\oplus,\circbin)\). Therefore, on the one hand,
Proposition~\ref{prop:monoid} and the identity in the proof of Theorem~\ref{thm:Tbrace} recover the corresponding
construction in~\cite[Example~3]{rumpnear}. On the other hand, Proposition~\ref{prop:Tastgroup} and Corollary~\ref{cor:T-top-group} complement Rump's results, providing sufficient conditions for the adjoint monoid \((\Id+T,\circbin)\) to be a topological group.

Rump also shows that every skew brace arises from an \emph{exponential} of a near-ring \(N\) with identity, namely an additive homomorphism \(e\colon N_c\to N\) such that \(e(r)\circbin 0=r\) for every \(r\in N_c\) and the elements \(e(r)+\Id\), with \(r\in N_c\), form a subgroup of the group of units of \(N\) \cite[Definition~6 and Theorem~1]{rumpnear}.
For the skew brace \((T,+,\ast)\) we may take \(N=M(T)\), the near-ring of all maps from $(T,+)$ to itself. Note that
each \(f\in T\) can be identified with the constant map of \(M(T)\) with value \(f\), so that \(M(T)_c\) is identified with \(T\). The exponential is then given by \(e(f)(g)=f\ast g-g\), and \(e(f)+\Id\) is the left translation \(g\mapsto f\ast g\) of \((T,\ast)\). Since \((T,\ast)\) is a group, these translations form a subgroup of the group of units of \(M(T)\), and \cite[Theorem~1]{rumpnear} applies.

Finally, recall that a near-ring \(N\) with identity is \emph{local}~\cite{Maxson} if the elements \(r\in N\) with \(N\circbin r\neq N\) form a submodule \(J(N)\) of \(N\), called the \emph{radical} of \(N\). In this case \(J(N)\) is the set of non-invertible elements of \(N\), and it is a skew brace with respect to the operation \(a\cdot b\defeq a\circbin(b+\Id)+b\) \cite[Example~8]{rumpnear}. It is therefore natural to ask whether \(T=J(N)\) for some local sub-near-ring \(N\) of \(R\) with \(\Id\in N\). If this is the case and \((N,+)\) is abelian, then \(a\cdot b=b+a\circbin(\Id+b)=a\ast b\), so that \((T,+,\ast)\) is the skew brace of \cite[Example~8]{rumpnear}, and the fact that \((T,\ast)\) is a group already follows from the locality of \(N\).
In Sections~\ref{sec:nottingham}--\ref{sec:ia-endo} we discuss, for each of our examples, whether such a near-ring \(N\) exists. 

\section{The Nottingham Group}\label{sec:nottingham}
Let $D[[x]]$ be the ring of formal power series over the commutative domain $D$,
whose elements are of the form
\[
f(x)=\sum_{n=0}^{\infty} a_nx^n, \qquad a_n \in D.
\]

\begin{definition}
	For \(f \in D[[x]]\), \(f\ne 0\), the \emph{\(x\)-adic order} \(\ord_x(f)\) of \(f\)
		is the largest integer \(n\ge 0\) such that \(x^n\) divides \(f\). The
		\emph{\(x\)-adic absolute value} of \(f\) is defined by
		\[
		|f|_x=2^{-\ord_x(f)}, \qquad |0|_x=0.
		\]
\end{definition}

The ring \(D[[x]]\) is naturally endowed with the topology induced by the \emph{\(x\)-adic
		ultrametric distance}
\[
d(f,g)=|f-g|_x.
\]

This topology can also be described as the one induced by the filtration
	\(\{x^n D[[x]]\}_{n\geq 0}\), taking these subgroups as a fundamental system of
	neighborhoods of \(0\); with respect to \(d\), the ring \(D[[x]]\) is a complete
	metric space.


The ideal \(R\defeq x D[[x]]\) is a \emph{right} near-ring under composition,
	with identity \(\Id=x\). As in Section~\ref{sec:general}, consider the additive
	subgroup \(T\defeq xR=x^2D[[x]]\) of \(R\), together with the decreasing
	\(x\)-adic filtration \(\{T_i\}_{i\geq 1}\), where
	\[
	T_i=x^iR= x^{i+1}D[[x]].
	\]
	Clearly \(T_1=T\), \(\bigcap_{i\ge 1}T_i=\{0\}\) and
	\(T_i \circbin (x+T) \subseteq T_i\) for all \(i \geq 1\), so that
	conditions~\eqref{eq:filtration} and~\eqref{eq:ast-closure-Ti} are satisfied,
	while condition~\eqref{eq:normality} is automatic since \((T,+)\) is abelian.
	Moreover, \(T\) is complete, being closed in \(D[[x]]\). In particular, by
	Proposition~\ref{prop:monoid}, \((T,\ast)\) is a monoid, where
	\[
	f\ast g=(x+f)\circbin(x+g)-x.
	\]


%



\begin{proposition}
The subgroup \(T\) is compatible.
\end{proposition}
\begin{proof}
	Let \(f=x^2\sum_{k\geq 0}a_kx^k\) and \(f'=x^2\sum_{k\geq 0}b_kx^k\) be elements
		of \(T\), and let \(g=x^{i+1}\sum_{k\geq 0}c_kx^k\in T_i\), with \(i\ge 1\). Then
		\begin{align*}
			\Delta_g(f,f')&=f\circbin(x+f'+g)- f\circbin(x+f') \\
			&= (x+f'+g)^2\sum_{k\geq 0}a_k(x+f'+g)^k-(x+f')^2\sum_{k\geq 0}a_k(x+f')^k  \\
			&\equiv (x+f')^2\Bigl(\sum_{k\geq 0}a_k(x+f'+g)^k-\sum_{k\geq 0}a_k(x+f')^k\Bigr)\\
			&\equiv (x+f')^2\Bigl(\sum_{k\geq 0}a_k(x+f')^k-\sum_{k\geq 0}a_k(x+f')^k\Bigr)\\
			&\equiv 0 \pmod{x^{i+2}},
		\end{align*}
		where the first congruence follows from
		\((x+f'+g)^2-(x+f')^2=2(x+f')g+g^2\equiv 0 \pmod{x^{i+2}}\), and the second one
		from the fact that \((x+f'+g)^k-(x+f')^k\) is divisible by \(g\) for all
		\(k\ge 1\), so that each such term, multiplied by \((x+f')^2\), lies in
		\(x^{i+3}D[[x]]\). Hence \(\Delta_g(f,f')\in x^{i+2}D[[x]]=T_{i+1}\).\qedhere
\end{proof}

By Theorem~\ref{thm:Tbrace}, the triple \((T,+,\ast)\) is a topological right
	brace. Moreover, since \((x+f)^k\equiv x^k \pmod{x^{k+1}}\) for every \(k\ge 1\),
	we have \(\ord_x(\Delta_f(h))>\ord_x(h)\) for all \(h\in T\), hence
	\(\lim_{i \to \infty}\Delta^i_f(f) = 0\) and Remark~\ref{rem:inverseformula}
	provides the closed formula
	\[
	(f)^{-1}_\ast=\sum_{i=0}^{\infty}(-1)^{i+1}\Delta^i_f(f).
	\]


We point out that \((T,\ast)\) is isomorphic to the Nottingham group over
\(D\) (see, e.g., \cite{Camina2000})
\[
\mathcal{N}(D) = \bigl\{ x+ x^2f \mid f \in D[[x]] \bigr\},
\]
whose operation is composition, via the map \(x^2f \mapsto x+x^2 f\). The same
	map endows \(\mathcal{N}(D)\) with the structure of a topological right
	brace.
	\bigskip
	
When \(D\) is a field, a series in \(R=xD[[x]]\) is invertible under composition if and only if its linear coefficient is non-zero. Hence the non-invertible elements of \(R\) form the subgroup \(T=x^2D[[x]]\), which satisfies \(R\circbin T\subseteq T\), so that \(R\) is local with \(J(R)=T\). Since \((R,+)\) is abelian, it follows from the discussion at the end of Section~\ref{sec:general} that in this case \((T,+,\ast)\) is the skew brace of \cite[Example~8]{rumpnear}. More generally, when \(D\) is a local domain with maximal ideal \(\mathfrak{m}\neq 0\), a series in \(R\) is invertible if and only if its linear coefficient lies in \(D\setminus\mathfrak{m}\). Thus \(R\) is still local, but \(J(R)=\mathfrak{m}x+T\), where \(\mathfrak{m}x=\{ax\mid a\in\mathfrak{m}\}\), properly contains \(T\).

	\section{Triangular Functions}\label{sec:triangular}
	
	Let \(A\) be an abelian group and let \(A_i^j \defeq\prod_{k=i}^j A\), with possibly \(i=-\infty\) and \(j=\infty\). We set
	\(A_i^j=\Set{0}\) if \(j<i\). For simplicity of notation we will denote by \(x_i^j=(x_i, \dots, x_j)\in A_i^j\) the
	truncation of the vector \(x\in A_{-\infty}^\infty\) to the interval of coordinates \([i,j]\). We denote by \(f_h\) the
	\(h\)-th component of a vector function \(f\). We define
	\begin{align}
		V_i^j&\defeq\Set{ f\colon A_i^j \to A},\notag\\
		U_i^j&\defeq \left\{f\colon A_i^j \to A_i^j \;\middle| \; f_h \in V_i^{h-1} \text{ for all } i\leq h \leq j\right\},\notag\\
		S_i &\defeq
		\left\{
		f\in U_{-\infty}^\infty \mid
		f_h=0 \text{ for } h\neq i
		\right\}. \label{eq:Si}
	\end{align}
	Moreover, we let \(T_i^j\) be an additive subgroup of \(U_i^j\) such that
		\(T_i^j=\prod_{h=i}^j \bigl(S_h\cap T_i^j\bigr)\).
	We point out that \(T_{i}^j\) can be equal to the whole group \(U_i^j=\prod_{h=i}^j S_h\). We use the symbol \(\Id\) to
	denote the identity function. By abuse of notation we consider \(T_{i}^j \subseteq T_{h}^k\) if \(h\le i\) and \(k\ge
	j\), by identifying two functions that are equal over their support. Note that under these conditions we have \(T_{i}^j \le
	T_{h}^k\). For the sake of simplicity of notation we also set \(T\defeq T_{-\infty}^\infty\), \(T_i\defeq T_i^\infty\) and
	\(T^i\defeq T_{-\infty}^i\). The functions in \(T\) are also known as \emph{triangular functions}, i.e.\ vector functions
	\(f\) in which each component \(f_h\) does not depend on the variables \(x_k\) with \(k\ge h\) (see also \cite{Klimov2002}, 
	where these functions were introduced for cryptographic applications).
	
	For \(f,g\in T_i^j\) we define
	\begin{equation}\label{eq:ast-tf}
		f\ast g \defeq(\Id+f) \circbin (\Id+g)-\Id=\oneplus{f}\circbin\oneplus{g}-\Id,
	\end{equation}
	where \(\oneplus f=\Id+f\), in accordance with~\eqref{eq:ast-via-Id}.
	Under this operation \((T_i^j, \ast)\) is a monoid with identity \(0\), isomorphic to \((\Id+T_i^j, \circbin)\).
	
	\subsection{Finite Triangular Functions}\label{sec:finitetriangular}
For \(1\le i\le n\) and \(y\in A\), we denote by \(ye_i\in A_1^n\) the vector whose
		\(i\)-th component is \(y\) and whose other components are \(0\), so that every
		\(x\in A_1^n\) can be written as \(x=x_1e_1+\dots+x_ne_n\).
	We define
	\[
	\mB_i \defeq\{\oneplus g  \mid g \in S_i\cap T_1^n\},\qquad 1\le i \le n.
	\]
	Note that \(\oneplus{-g}\) is the inverse function of \(\oneplus{g}\in \mB_i\). Indeed,
	since \(g(x+ye_i)=g(x)\) for all \(x\in A_1^n\) and all \(y\in A\), we have \((\Id+g)\circbin (\Id-g)=\Id\). In particular
	every function \(\oneplus{g}\in \mB_i\) is invertible, and \(\mB_i\) is a subgroup of the group of invertible
		functions from \(A_1^n\) to itself under composition.
	
	\begin{lemma}
		\([\mB_i,\mB_j] \le \mB_{\max(i,j)}\) for all \(1\le i,j\le n\).
	\end{lemma}
	
	\begin{proof}
	If \(i=j\) the claim is trivial, since \(\mB_i\) is abelian; hence assume \(j<i\).
		Let \(x=x_1e_1+\dots+x_ne_n\in A_1^n\). Consider
		\(\oneplus g\in \mB_i\) and \(\oneplus h\in \mB_j\). Writing
			\([a,b]=a\circbin b\circbin a^{-1}\circbin b^{-1}\), an easy calculation
			shows that
			\begin{multline*}
				[\oneplus g,\oneplus h](x)=x+g(x_1e_1+\dots+x_{i-1}e_{i-1})\\
				-g\bigl(x_1e_1+\dots+x_{j-1}e_{j-1}+x_je_j-h(x_1e_1+\dots+x_{j-1}e_{j-1})
				+x_{j+1}e_{j+1}+\dots+x_{i-1}e_{i-1}\bigr),
			\end{multline*}
			that is, \([\oneplus g,\oneplus h]=\Id-\Delta_{-h}g=\oneplus{-\Delta_{-h}g} \in \mB_i\).
	\end{proof}
	\begin{remark}\label{rem:wreath}
		Note that \(W_n\defeq\mB_1 \ltimes \cdots \ltimes \mB_n=\Id+T_{1}^n\). When
			\(T_1^n=U_1^n\), the group \(W_n\) is the iterated wreath product
			\(\underbrace{A\wr\dots \wr A}_{\text{\(n\) times}}\), where \(\mB_i\) is the
			\(i\)-th base subgroup (see also~\cite{Aragona2026}).
	\end{remark}
	
	Observe that \(\Id \notin T_1^n\) and that $T_1^n$ is a subnear-ring of the right near-ring of functions from \(A_1^n\)
	to \(A_1^n\) under pointwise sum and composition. We now introduce a filtration which fits into the general framework discussed
	above.
	
	Consider the descending chain \(\{T_i^n\}_{1\le i \le n}\) of additive
		sub-near-rings of \(T_1^n\), extended by setting \(T_i^n=\{0\}\) for \(i>n\).
		Note that
		\begin{gather*}
			\bigcap_{i\geq 1}T_i^n=\{0\},\\
			T_i^n \circbin (\Id+T_1^n) \subseteq T_i^n \quad\text{for all } i\ge 1,
		\end{gather*}
		so that conditions~\eqref{eq:filtration} and~\eqref{eq:ast-closure-Ti} hold, while
		condition~\eqref{eq:normality} is automatic since \((T_1^n,+)\) is abelian. This
		finite filtration is therefore a particular case of the general framework of
		Section~\ref{sec:general}; the induced topology is discrete, and in particular
		\(T_1^n\) is complete.
	\begin{proposition}\label{prop:T1ncomp}
		The near-ring \(T_1^n\) is compatible.
	\end{proposition}
	\begin{proof}
	If \(f, f'\in T_1^n\) and \(g\in T_i^n\), then for every \(1\le k\le n\) and
			every \(x\in A_1^n\) we have
			\begin{align*}
				\Delta_g(f,f')_k(x)&=\bigl(f_k\circbin(\Id+f'+g)\bigr)(x)- \bigl(f_k\circbin(\Id+f')\bigr)(x) \\
				&= f_k\bigl((x+f'(x)+g(x))_{1}^{k-1}\bigr)- f_{k}\bigl((x+f'(x))_{1}^{k-1}\bigr).
			\end{align*}
			If \(k\leq i\), then \((x+f'(x)+g(x))_{1}^{k-1}=(x+f'(x))_{1}^{k-1}\), since the
			components \(g_h\) vanish for \(h\le k-1<i\); hence \(\Delta_g(f,f')_k=0\).
			Thus \(\Delta_g(f,f')\in T_{i+1}^n\).
	\end{proof}
	
	Note that \(W_n=\Id+T_1^n\) is a group under composition; hence \((T_1^n,\ast)\)
		is a group, isomorphic to \(W_n\). Therefore, by Proposition~\ref{prop:T1ncomp}
		and Theorem~\ref{thm:Tbrace}, \((T_1^n, +, \ast)\) is a right brace.
	
	\bigskip
	
	We proceed by giving an explicit formula for the inverse with respect to \(\ast\) of an element in \(T_1^n\).
	For each $1 \le i \le n$, let us consider \(G_i \in S_i \cap T_1^n\). Observe that
	\[
	\Delta_{G_i}(G_i)= G_i\circbin (\Id+G_i) -G_i=0, 
	\]
	meaning that, in this case, the difference operator \(\Delta\) is nilpotent already at step one. Therefore we can apply
	the inverse formula in Remark~\ref{rem:inverseformula}
\[
s(G_i)= \sum_{k=0}^{\infty} (-1)^{k+1} \Delta^k_{G_i}(G_i)= -G_i,
\]
	and trivially
	\[
	G_i \ast (-G_i)= -G_i + G_i \circbin (\Id -G_i)= -G_i+ G_i =0.
	\]
	Now, \(\bigcup_{1\le i \le n}\bigl(S_i \cap T_1^n\bigr)\) is a generating set for
		\((T_1^n,\ast)\): indeed, let \(g\in T_1^n\) and, for \(1\le i\le n\), let
		\(G_i\in S_i\cap T_1^n\) be the function whose \(i\)-th component is \(g_i\).
	We have
	\begin{align*}
		\oneplus{g}(x) &= x + g(x)\\
		&= \big(x_1 + g_1(0),\, x_2+g_2(x_1),\, \dots ,\, x_n + g_n(x_1, \dots, x_{n-1})\big) \\
		&= (\Id+G_1)\circbin (\Id+G_2) \circbin \dots \circbin (\Id + G_n)(x),
	\end{align*}
	so that \(g = G_1 \ast \dots \ast G_n\).
	Hence \(h\defeq(-G_n)\ast (-G_{n-1}) \ast \dots \ast (-G_1)\) is the inverse of \(g\) with respect to \(\ast\).
	Moreover, by a direct calculation, we can see that for every \(g \in T_1^n\) the formula in~\eqref{eq:delta} becomes
	\begin{equation}\label{eq:delta-explicit}
		\Delta^i_g(g) = \sum_{j=0}^{i+1}(-1)^j \binom{i+1}{j}(\oneplus{g})^{(i+1)-j},
	\end{equation}
	where the powers of \(\oneplus{g}\) are taken with respect to \(\circbin\) and
		\((\oneplus{g})^{0}=\Id\).
		
\bigskip

Assume that \(A\) is an elementary abelian \(p\)-group, and let \(N\defeq\mathbb{Z}\Id+T_1^n=\mathbb{F}_p\Id+T_1^n\), which is a sub-near-ring of the near-ring of maps from \(A_1^n\) to itself. Every element \(\lambda\Id+t\) of \(N\) with \(\lambda\neq 0\) is invertible, as one checks by solving recursively on the coordinates. On the other hand, no element \(t\in T_1^n\) depends on \(x_n\), so that \(t\) is not injective and hence not invertible. Therefore the non-invertible elements of \(N\) form the subgroup \(T_1^n\), and since \(N\circbin T_1^n\subseteq T_1^n\), the near-ring \(N\) is local with \(J(N)=T_1^n\). As \((N,+)\) is abelian, the discussion at the end of Section~\ref{sec:general} shows that, as for the Nottingham group over a field, \((T_1^n,+,\ast)\) is the skew brace of \cite[Example~8]{rumpnear}. More generally, if \(A\) is an abelian \(p\)-group of exponent \(p^k\) with \(k\ge 2\), then an element \(\lambda\Id+t\) of \(N\) is invertible if and only if \(p\nmid\lambda\). Thus \(N\) is still local, but \(J(N)=p\mathbb{Z}\Id+T_1^n\) properly contains \(T_1^n\).

	\subsection{Finitary Triangular Functions}
	When \(A\) is a field and \(T_i^j\) is a group of linear functions, the sets \(\Id+T_i^j\) are finitary linear groups (see, e.g.,~\cite{Hall1995}). This allows us to call \emph{finitary triangular maps} the elements \(F\in
	T_i^j\), where \(i\) and \(j\) are finite. If \(F\in T_i^j\) and \(G\in T_h^k\) are two finitary triangular
	maps, then \(F, G, F \ast G, F + G\in T_u^v\), where \(u=\min (i,h)\) and \(v=\max(j,k)\). Hence, by the
	previous subsection, the set \(\mathcal{FT}\defeq\bigcup_{-\infty<i\le j<\infty} T_i^j\) is a right brace with respect to \(+\) and \(\ast\), 
	called the brace of \emph{finitary triangular functions}. As observed in Remark~\ref{rem:wreath}, the group \((\Id+T^j_i,\circbin)\) is
	an iterated wreath product of a finite number of copies of \(A\). Suppose that \(A\) is finite. Then the groups
		\((T^j_i,\ast)\) are soluble, and hence \(\mathcal{FT}\) is locally soluble. Moreover, \(\mathcal{FT}\) is locally nilpotent if and
		only if \(A\) is a \(p\)-group, by~\cite{Baumslag1959}.

		\subsection{Positive Triangular Functions}
		This subsection is devoted to the case \(A_0^{\infty}= \prod_{i=0}^{\infty}A\). Using our notation, in this subsection
		\(T_i\) is the additive group of all functions \(f\colon A_0^{\infty} \to A_0^{\infty}\) such that
		\begin{itemize}
			\item for \(h\ge i\) the component \(f_h\) depends only on the variables \(x_k\) with \(0 \leq k\leq h-1\);
			\item \(f_h=0\) for all \(0\leq h < i \).
		\end{itemize}
		We recall that in our notation \(x_i^j=(x_i, \dots, x_j)\in A_i^j\) denotes the truncation of the vector \(x\)
		to the interval of coordinates \([i,j]\). We consider the following filtration, satisfying~\eqref{eq:filtration},
		\[
		T_0 \supseteq T_1\supseteq \dots \supseteq T_i \supseteq T_{i+1} \supseteq\cdots
		\]	
		If \(f \in T_i\) and \(g \in T_0\), then
		\begin{align*}
				\bigl(f\circbin (\Id+g)\bigr)(x)&= f\bigl(x_0 + g_0(0),\, x_1 +g_1(x_0),\, x_2+g_2(x_0,x_1),\, \dots \bigr)\\
				&=\bigl(0, \dots,0,\, f_i\bigl((x+g(x))_0^{i-1}\bigr),\, f_{i+1}\bigl((x+g(x))_0^{i}\bigr),\, \dots\bigr),
			\end{align*}
			whose components of index \(<i\) vanish; thus \(f\circbin (\Id+g)\in T_i\), i.e.,
			Condition~\eqref{eq:ast-closure-Ti} is satisfied, while Condition~\eqref{eq:normality}
			is automatic since \((T_0,+)\) is abelian.
		
		As in the general framework described in Section~\ref{sec:general}, the filtration \(\{T_i\}_{i\ge 0}\) induces a
		topology on \(T_0\) where the subgroups \(T_i\) form a fundamental system of neighborhoods of \(0\). In the same way we can
		define the distance \(d\), which turns $T_0$ into a complete ultrametric space.
		
		\begin{proposition}
			The near-ring \(T_0\) is compatible.
		\end{proposition}
		\begin{proof}
		Let \(f,f'\in T_0\) and \(g\in T_i\). Since \(g_h=0\) for \(h<i\), the vectors
				\(x+f'(x)+g(x)\) and \(x+f'(x)\) agree in the coordinates \(0,\dots,i-1\); hence
				\begin{align*}
					\Delta_g(f,f')(x) & =f\bigl(x+f'(x)+g(x)\bigr)- f\bigl(x+f'(x)\bigr)\\
					&= \Bigl(0, \dots, 0,\, f_{i+1}\bigl((x+f'(x)+g(x))_0^{i}\bigr) -  f_{i+1}\bigl((x+f'(x))_0^{i}\bigr),\, \dots\Bigr),
				\end{align*}
				where all the components of index \(\le i\) vanish. Thus \(\Delta_g(f,f')\in T_{i+1}\).
		\end{proof}
		By Theorem~\ref{thm:Tbrace} we have the following result.
		\begin{proposition}
			The triple \((T_0,+,\ast)\) is a topological right brace.
		\end{proposition}
		With \(S_i\) defined as in~\eqref{eq:Si}, we can summarize the following properties of \(T_0\).
		\begin{proposition}
		For every \(i\ge 0\) and \(j\ge 1\) the following hold:
			\begin{enumerate}
				\item \((T_{i}, \ast) \trianglelefteq (T_0, \ast)\),
				\item \(T_i=T_{i+1} \rtimes (S_i\cap T_0)\),
				\item \(T_0^j= (S_{j} \cap T_0)\rtimes \dots \rtimes (S_0\cap T_0)\),
				\item \(T_0= T_{j} \rtimes \bigl((S_{j-1} \cap T_0) \rtimes \dots \rtimes (S_0 \cap T_0)\bigr)   = T_j \rtimes  T_{0}^{j-1} \).
			\end{enumerate}
		\end{proposition}
		
		\begin{remark} Note that
			\[
			T_0^0 \subseteq T_0^1 \subseteq \dots \subseteq T_0^{i} \subseteq T_0^{i+1} \subseteq \dots
			\]
			and that \(T_0\) is the closure of the subgroup \(\bigcup_{j \geq 0} T_0^j \) of finitary functions.
			Moreover, by the previous proposition, \(T_0\) is the inverse limit
			\[
				T_0=\varprojlim T_0/T_k =\varprojlim  T_0^{k}.
				\]
		\end{remark}
		
		\subsection{Negative Triangular Functions}
		In this subsection we consider, for \(i\le 0\), the additive groups \(T^i\) of all functions \(f\colon A_{-\infty}^{0} \to A_{-\infty}^{0}\) such that
		\begin{itemize}
			\item for \(h\le 0\) the component \(f_h\) depends only on the variables \(x_k\) with \(k\leq h-1\);
			\item \(f_h=0\) for all \(h>i \).
		\end{itemize}
		This case is slightly different from the previous ones. Indeed \(T^0\) is not a group with respect to \(\ast\). The
		function \(f\) defined by \(f_i(x)=x_{i-1}\) provides an example for which \(\Id+f\) is not invertible.
		
		This is easily seen by fixing \(0\ne h\in A\) and setting \(b_i\defeq a_i+(-1)^ih\), where \(a\in A_{-\infty}^0\). We have
		\begin{align*}
			\bigl((\Id+f)(b)\bigr)_i= b_i+f_i(b) &=b_i+b_{i-1}
			=a_i+(-1)^ih+a_{i-1}+(-1)^{i-1}h\\ & =a_i+a_{i-1}= a_i+f_i(a) =\bigl((\Id+f)(a)\bigr)_i,
		\end{align*}
		i.e., \((\Id+f)(b)=(\Id+f)(a)\), with \(a\ne b\) since \(h\ne 0\), and thus \(\Id+f\) is not injective. In particular \(f\)
		has no inverse with respect to \(\ast\).
		
		Since \(f\) is linear, we have that there is no hope of defining a brace structure even in the case of upper triangular
		matrices of the form \((1+a_{i,j})_{-\infty < i< j \le 0}\). Nevertheless \(T^0\) is a monoid with respect to
		\(\ast\).
		
		We consider the following filtration
		\begin{equation}\label{eq:filt}
			T^{0} \supseteq T^{-1} \supseteq\dots \supseteq T^{j} \supseteq  T^{j-1}\supseteq \cdots
		\end{equation}
		where trivially \(\bigcap_{ j \leq 0 }T^{j}=\{0\}\).
		The filtration \eqref{eq:filt} turns \((T^0, +)\) into a topological group. We point out that \(T^i\) and \(S_i\cap
		T^0\) are closed with respect to the operation \(\ast\) and that we have the semidirect product decomposition of monoids
		\[
		T^j= T^{j-1} \ltimes_{\ast} (S_j\cap T^0).
		\]
		Indeed \(T^{j-1} \cap (S_j\cap T^0)=\Set{0}\) and, if \(s\in S_j\cap T^0\) and
		\(f\in T^{j-1}\), setting \(s'(x)=s(x+f(x))\) we obtain
		\begin{align*}
			(s \ast f) (x) &= \bigl((\Id+s)\circbin(\Id+f)-\Id\bigr)(x)\\
			&= \bigl(f+s\circbin (\Id+f)\bigr)(x) = f(x)+s(x+f(x)) = (f\ast s')(x).
		\end{align*}
		In particular, if \(f,g\in T^{j-1}\) and \(s,t\in S_j\cap T^0\), then \((g\ast s)\ast (f\ast t) = (g\ast f)\ast (s'\ast t)\),
		where \(g\ast f\in T^{j-1}\) and \(s' \ast t \in S_j\cap T^0\).

		So, with respect to the \(\ast\)-operation, for every \(k \ge 0\) it holds
			\begin{align*}
				T^j &= T^{j-1} \ltimes (S_j \cap T^0) \\
				&= T^{j-k} \ltimes\bigl((S_{j-k+1} \cap T^0) \ltimes \dots \ltimes (S_j\cap T^0)\bigr).
			\end{align*}
			In particular, taking \(j=0\), for every \(i\le 0\) we have
			\[	T^0 = T^{i} \ltimes\bigl((S_{i+1} \cap T^0) \ltimes \dots \ltimes (S_0\cap T^0)\bigr).\]
			Setting \(\Sigma_i\defeq (S_{i+1} \cap T^0) \ltimes \dots \ltimes (S_0\cap T^0)\), we have \(\Sigma_i\cap T^i=\Set{0}\); hence
			each \(T^i\) is a complement of the normal sub-semigroup \(\Sigma_i\) in \(T^0=T^i\ltimes \Sigma_i\).
		
		\section{Skew braces from iterated wreath products}\label{sec:wreath}
		Let $n\ge1$ and let $G_1,\dots,G_n$ be (not necessarily abelian) groups. Set
		$G\defeq G_1\times\dots\times G_n$ and let $R\defeq M(G)$ be the right
		near-ring of all maps $G\to G$, where $+$ is the pointwise multiplication of
		$G$ and $\circbin$ is the composition of maps; the identity element of
		$(R,\circbin)$ is the identity map $\Id$.
		
		For $1\le i\le n$ set $B_i\defeq G_i^{G_1\times\dots\times G_{i-1}}$, where
		for $i=1$ the empty product is a singleton, so that $B_1\cong G_1$. We define
		$T$ as the set of maps whose $i$-th component only depends on the variables
		$x_1,\dots,x_{i-1}$, namely
		\[
		T\defeq\bigl\{f\in R \mid
		f(x_1,\dots,x_n)=\bigl(f_1,\,f_2(x_1),\dots,f_n(x_1,\dots,x_{n-1})\bigr),
		\ f_i\in B_i\bigr\}.
		\]
		Since $+$ is computed componentwise and pointwise, $(T,+)$ is a subgroup of
		$(R,+)$, isomorphic to the direct product $\prod_{i=1}^{n}B_i$. For
		$1\le j\le n+1$ we set
		\[
		T_j\defeq\{f\in T \mid f_i=1 \text{ for all } i<j\}
		\cong\prod_{i=j}^{n}B_i,
		\]
		and $T_j\defeq\{0\}$ for $j>n$. This is a finite decreasing chain of direct
		factors of $(T,+)$ with trivial intersection, so it is a filtration
		satisfying conditions~\eqref{eq:filtration} and~\eqref{eq:normality}.
		\begin{proposition}
			$T$ is compatible with respect to the filtration \(\{T_i\}_{1\le i\le n+1}\).
		\end{proposition}
		
		\begin{proof}
			First observe that for every $g\in T$ the $m$-th component of $(\Id+g)(x)$
			is
			\[
			x_m\,g_m(x_1,\dots,x_{m-1}),
			\]
			which only depends on $x_1,\dots,x_m$.
			
			Let $f\in T_j$ and $g\in T$. The $i$-th component of $f\circbin(\Id+g)$
			is $f_i$ evaluated at the first $i-1$ components of $(\Id+g)(x)$: it is
			trivial for $i<j$ and it only depends on $x_1,\dots,x_{i-1}$. Hence
			$f\circbin(\Id+g)\in T_j$, and so Condition~\eqref{eq:ast-closure-Ti} holds.
			
			Let $f,f'\in T$ and $g\in T_j$. It remains to prove that $\Delta_g(f,f')\in T_{j+1}$, i.e.,
			$\Delta_g(f,f')_i=1$ for every $i\le j$. Write
			\[
			u=(\Id+f'+g)(x),
			\qquad
			v=(\Id+f')(x),
			\]
			so that the $i$-th component of $\Delta_g(f,f')$ is
			\[
			\Delta_g(f,f')_i(x)
			=f_i(u_1,\dots,u_{i-1})\,f_i(v_1,\dots,v_{i-1})^{-1}.
			\]
			Since \(g_m=1\) for every \(m<j\), we have that \(u,v\) agree on the first \(j-1\) components.
			Now fix $i\le j$. The map $f_i$ only reads the components
			$u_1,\dots,u_{i-1}$ of its argument, and $i-1\le j-1$, so
			$f_i(u_1,\dots,u_{i-1})=f_i(v_1,\dots,v_{i-1})$ and
			$\Delta_g(f,f')_i=1$, as claimed. Therefore
			$\Delta_g(f,f')\in T_{j+1}$.
		\end{proof}
		Since the filtration is finite, the induced topology is discrete and $T$ is
		complete. Theorem~\ref{thm:Tbrace} then yields the following.
		
		\begin{corollary}\label{cor:wreath}
			$(T,+,\ast)$ is a topological right skew brace with the discrete
			topology. Its additive group is the direct product
			$\prod_{i=1}^{n}B_i$, while the map $f\mapsto\Id+f$ is an isomorphism
			from $(T,\ast)$ onto the group of permutations of $G$ of the form
			\[
			(x_1,\dots,x_n)\longmapsto
			\bigl(x_1f_1,\;x_2f_2(x_1),\dots,\;x_nf_n(x_1,\dots,x_{n-1})\bigr),
			\]
			which is the iterated wreath product $G_n\wr\dots\wr G_1$ with
			respect to the regular actions.
		\end{corollary}
		Recall that the additive group of a finite local near-ring is a \(p\)-group for some prime \(p\) (see~\cite[Corollary~7.6]{Maxson}). Assume that \(G_1,\dots,G_n\) are finite, but not all \(p\)-groups for the same prime \(p\). Then \((T,+)=\prod_{i=1}^{n}B_i\) is not a \(p\)-group, since each \(B_i\) is a direct product of copies of \(G_i\). Hence \(T\) is not contained in any finite local near-ring \(N\), and in particular \((T,+,\ast)\) cannot be obtained as the radical \(J(N)\) of a finite local near-ring \(N\), nor as a sub-skew brace of it, as in \cite[Example~8]{rumpnear}.

		\section{IA-endomorphisms of free nilpotent groups}\label{sec:ia-endo}
		Let \(n\ge 2\) and let \(\Phi\) be the free group on \(x_1,\dots,x_n\). We denote by \(\gamma_i(\Phi)\) the \(i\)-th
		term of the lower central series of \(\Phi\) and we set
		\[
		F=F_{n,c}=\Phi/\gamma_{c+1}(\Phi)
		\]
		the free nilpotent group of class \(c\) on \(n\) generators. We still write \(x_1,\dots,x_n\) for the images of the
		generators.
		\begin{remark}
			We recall that, since \(F\) is free in the variety of nilpotent groups of class at most \(c\), every map \(\theta_0\colon \{x_1,\dots,x_n\}\to F\)
			uniquely extends to an endomorphism \(\theta\) of \(F\) such that \(\theta(x_i)=\theta_0(x_i)\) for each
			\(i=1,\dots,n\).
		\end{remark}
		\begin{definition}
			For \(\varphi,\psi\in \End(F)\) we define \(\varphi+\psi\) as the unique endomorphism of \(F\) such that
			\[
			(\varphi+\psi)(x_i)=\varphi(x_i)\psi(x_i)
			\]
			for all \(i=1,\dots,n\).
		\end{definition}
		Set \(R\defeq\End(F)\). It is straightforward to check that \((R,+)\) is a group isomorphic to \(F^n\). We stress that the identity
			\((\varphi+\psi)(w)=\varphi(w)\psi(w)\) holds when \(w=x_i\) is a generator, but it fails for a general word \(w\in F\).
		\begin{lemma}
			The triple \((R,+,\circbin)\) is a left near-ring with identity \(\Id\).
		\end{lemma}
		From now on we write \(\gamma_s\defeq \gamma_s(F)\); note that \(\gamma_{c+1}=1\). We consider
		\[
		T=\{\varphi\in R\mid \varphi(x_j)\in \gamma_2 \text{ for all }j\}
		\]
		and the filtration \(\{T_i\}_{i\ge 1}\) defined as
		\[
		T_i=\{\varphi\in R\mid \varphi(x_j)\in \gamma_{i+1}\text{ for all }j\}.
		\]
		In this notation \(T_1=T\) and \(T_c=\{0\}\), since \(\gamma_{c+1}=1\); hence the filtration is finite, the induced topology
		is discrete and the completeness of \(T\) is automatic.
		\begin{definition}\cite{Bachmuth1965} 
			An endomorphism \(\varphi\) of a group \(G\) is said to be an \(IA\)-endomorphism if it induces the identity on the
			abelianization \(G/\gamma_2(G)\).
		\end{definition}
		It is worth noting that, in our setting, the set of \(IA\)-endomorphisms of
			\(F\) is precisely \(\Id+T\).
		
		\begin{remark}
			Note that \((T,+)\cong (\gamma_2)^n\) and, in general, \((T_i,+)\cong (\gamma_{i+1})^n\). Moreover
				\(T_i\trianglelefteq (T,+)\): indeed, if \(\eta\in T\) and \(\varphi\in T_i\), then
				\[
				(-\eta+\varphi+\eta)(x_j)=\eta(x_j)^{-1}\varphi(x_j)\eta(x_j)\in\gamma_{i+1},
				\]
				since \(\gamma_{i+1}\) is normal in \(F\). Finally, \((\Id+T)\circbin T_i\subseteq T_i\), since
				\(\gamma_{i+1}\) is fully invariant in \(F\); this is the mirrored closure condition of
				Remark~\ref{rem:left_near_ring}.
		\end{remark}
		
		\begin{theorem}[\cite{Andreadakis1965}]\label{thm:andreadakis}
			For all \(i,j\ge 1\) one has \([\gamma_j,\Id+T_i]\subseteq\gamma_{i+j}\).
		\end{theorem}
		\begin{lemma}\label{lem:IA-invertible}
			Every \(\alpha\in\Id+T\) is an automorphism of \(F\).
		\end{lemma}
		\begin{proof}
			By Theorem~\ref{thm:andreadakis} we have that, for every \(\sigma\in\Id+T\) and every \(v\in \gamma_s\),
			\begin{equation}\label{eq:andreadakis}
				\sigma(v)v^{-1}\in\gamma_{s+1}.
			\end{equation}
			Since \(\alpha\) is an \(IA\)-endomorphism, every \(x\in F\) can be written as \(x=\alpha(x)z\) with
			\(z\in \gamma_2\). Iterating this process, using Equation~\eqref{eq:andreadakis} and the fact that \(\gamma_{c+1}=1\),
				we obtain the surjectivity of \(\alpha\). For the injectivity, suppose that \(\alpha(x)=1\) with
			\(1\neq x\in\gamma_s\setminus \gamma_{s+1}\).
			By Equation~\eqref{eq:andreadakis} we have that
			\(\alpha(x)x^{-1}\in\gamma_{s+1}\) and so \(x^{-1}\in \gamma_{s+1}\), a contradiction.
		\end{proof}
		
		We recall that in the case of left near-rings (see Remark~\ref{rem:left_near_ring})
		\(\Delta_g(f,f')=(\Id+f'+g)\circbin f-(\Id+f')\circbin f\).
		
		\begin{corollary}
			Let \(f\in T_j\), \(f'\in T\) and \(g\in T_i\). Then \(\Delta_g(f,f')\in T_{i+j}\).
		\end{corollary}
		
		\begin{proof}
			Set \(\alpha\defeq\Id+f'\) and \(\beta\defeq\Id+f'+g\). We note that \(\alpha\) is invertible by Lemma~\ref{lem:IA-invertible}. We claim that \(\beta\circbin\alpha^{-1}\in \Id+T_i\), i.e.\
			\((\beta\circbin\alpha^{-1})(u)u^{-1}\in\gamma_{i+1}\) for every \(u\in F\).
			Notice that the set \(\{v\in F\mid \beta(v)\alpha(v)^{-1}\in\gamma_{i+1}\}\) is a subgroup of \(F\)
			by normality of \(\gamma_{i+1}\), and it contains the generators, since
			\[
			\beta(x_k)\alpha(x_k)^{-1}
			= x_k f'(x_k)\,g(x_k)\,\bigl(x_k f'(x_k)\bigr)^{-1}
			= g(x_k)^{(x_k f'(x_k))^{-1}}\in\gamma_{i+1}.
			\]
			Hence it is the whole of \(F\). Choosing \(v=\alpha^{-1}(u)\) we get \((\beta\circbin \alpha^{-1})(u)u^{-1}\in\gamma_{i+1}\), which proves the claim.
			
			Now let \(x_m\) be a generator and set \(w\defeq f(x_m)\in\gamma_{j+1}\). Writing
				\(\beta(w)\alpha(w)^{-1}=\sigma(u)u^{-1}\) with \(\sigma\defeq\beta\circbin\alpha^{-1}\in\Id+T_i\) and
				\(u\defeq\alpha(w)\in\gamma_{j+1}\), Theorem~\ref{thm:andreadakis} yields
			\[
			\Delta_g(f,f')(x_m)
			=\beta\bigl(f(x_m)\bigr)\,\alpha\bigl(f(x_m)\bigr)^{-1}
			\in[\gamma_{j+1},\Id+T_i]\subseteq\gamma_{i+j+1}.
			\]
			It follows that \(\Delta_g(f,f')\in T_{i+j}\).
		\end{proof}
		In particular, taking \(j=1\), the near-ring \(T\) is compatible. Since the filtration
			is finite, \(T\) is complete with respect to the discrete topology; hence, by
			Remark~\ref{rem:left_near_ring}, \((T,+,\ast)\) is a topological left
			skew brace.
			
			\begin{remark}
				By Lemma~\ref{lem:IA-invertible}, every \(IA\)-endomorphism of \(F\) is in fact an
				automorphism, so that \(\Id+T\) coincides with the group of
				\(IA\)-automorphisms of \(F\). Since \(\Id+(f\ast g)=(\Id+f)\circbin(\Id+g)\),
				the map \(f\mapsto\Id+f\) is an isomorphism between \((T,\ast)\) and the group
				of \(IA\)-automorphisms of \(F\).
			
				The additive group \((T,+)\cong(\gamma_2)^n\) is non-abelian if and only if
				\(\gamma_2\) is non-abelian. Since \([\gamma_2,\gamma_2]\subseteq\gamma_4\),
				for \(c\le 3\) the group \(\gamma_2\) is abelian and \((T,+,\ast)\) is a left
				brace; hence, in order to obtain a skew brace which is not a brace, one needs
				\(c\ge 4\).
		\end{remark}

	\section*{Acknowledgements}
The authors  acknowledge  the funding support from MUR (Italy) through the PRIN 2022 project 2022RFAZCJ,
\emph{Algebraic methods in Cryptanalysis}.
The authors are grateful to the anonymous referee for the careful reading of the manuscript and for valuable comments, and in particular for drawing their attention to Rump's work~\cite{rumpnear} on the connection between near-rings and skew braces, which led to a clearer presentation of our results.

 \bibliography{citation}

\end{document}